\let\chooseClass2   %English, wide

\documentclass[12pt,a4paper,reqno]{amsart}

    \ifx\chooseClass1
\usepackage[cp1251]{inputenc}
\usepackage[T2A]{fontenc}
\newlength{\Totalheight}
    \else
    \fi

\usepackage{amssymb,amsxtra,mathtools}

\newtheoremstyle{boldhead}%     name
{\topsep}%                      abovespace
{\topsep}%                      belowspace
{\slshape}%                     bodyfont
{}%                             indentation=noindent
{\bfseries}%                    headfont
{.}%                            headpunctuation
{ }%                            headspace=interword space
{\thmname{#1}\thmnumber{ #2}\thmnote{ (#3)}}%   custom head specification

\newtheoremstyle{boldremark}%   name
{\topsep}%                      abovespace
{\topsep}%                      belowspace
{\upshape}%                     bodyfont
{}%                             indentation=noindent
{\bfseries}%                    headfont
{.}%                            headpunctuation
{ }%                            headspace=interword space
{\thmname{#1}\thmnumber{ #2}\thmnote{ (#3)}}%   custom head specification

\swapnumbers

\theoremstyle{boldhead}
\newtheorem{theorem}[subsection]{Theorem}
\newtheorem{lemma}[subsection]{Lemma}
\newtheorem{proposition}[subsection]{Proposition}

\theoremstyle{boldremark}
\newtheorem{definition}[subsection]{Definition}
\newtheorem{remark}[subsection]{Remark}

\numberwithin{equation}{section}

\usepackage{ifpdf}
\ifpdf
\usepackage[pdftex]{graphics}
\IfFileExists{hyperref.sty}{\usepackage[pdftex]{hyperref}}{}
\else
\usepackage[dvips]{graphics}
\IfFileExists{hyperref.sty}{\usepackage[hypertex]{hyperref}}{}
\fi

\usepackage[mathcal]{euscript}
\newcommand{\ca}{{\mathcal A}}
\newcommand{\cb}{{\mathcal B}}
\newcommand{\cc}{{\mathcal C}}
\newcommand{\cd}{{\mathcal D}}
\newcommand{\ce}{{\mathcal E}}

\newcommand{\cn}{{\mathcal N}}

\newcommand{\cp}{{\mathcal P}}
\newcommand{\cq}{{\mathcal Q}}

\usepackage{mathrsfs}
\newcommand\cQuiver{{\mathscr Q}}

\let\eps\varepsilon
\let\epsilon\varepsilon
\let\ge\geqslant
\let\kk\Bbbk
\let\le\leqslant
\let\rto\xrightarrow
\let\tens\otimes
\let\wt\widetilde

\newcommand{\bj}{{\mathbf j}}
\newcommand{\bull}{{\scriptscriptstyle\bullet}}

\newcommand{\McD}{{\mathscr D}}
\newcommand{\ucom}{\underline{\mathsf C}_\kk}
\newcommand{\uCom}{{\underline{\mathsf C}_\kk}}
\newcommand{\su}{{\mathsf{su}}}
\newcommand{\tdt}{\otimes\dots\otimes}
\newcommand{\uni}{{\mathbf i}}
\newcommand\ZZ{{\mathbb Z}}

\newcommand{\sS}[2]{\vphantom{#2}#1 #2}
\newcommand{\n}[1]{\nobreakdash-\hspace{0pt}}
\newcommand{\ainf}[1]{$A_\infty$\nobreakdash-\hspace{0pt}}

\DeclareMathOperator\Cone{Cone}
\DeclareMathOperator\id{id}

\DeclareMathOperator\inj{in}
\DeclareMathOperator\Ob{Ob}
\DeclareMathOperator\pr{pr}

\newcommand{\defeq}{\overset{\textup{def}}{=}}

\newcommand{\defref}[1]{Definition~\ref{#1}}

\newcommand{\lemref}[1]{Lemma~\ref{#1}}
\newcommand{\propref}[1]{Proposition~\ref{#1}}

\newcommand{\secref}[1]{Section~\ref{#1}}
\newcommand{\thmref}[1]{Theorem~\ref{#1}}

\begin{document}

\title{Unital $A_\infty$-categories}
\author{Volodymyr Lyubashenko and Oleksandr Manzyuk}


\begin{abstract}
    \ifx\chooseClass1
Ми доводимо, що три означення унітальності для \ainf-категорій
запропоновані Любашенком, Концевичем і Сойбельманом, та Фукая є
еквівалентними.

\bigskip
    \fi

We prove that three definitions of unitality for \ainf-categories
suggested by the first author, by Kontsevich and Soibelman, and by Fukaya
are equivalent.
\end{abstract}

\maketitle

\allowdisplaybreaks[1]

\section{Introduction}

Over the past decade, \ainf-categories have experienced a
re\-sur\-gence of interest due to applications in symplectic geometry,
deformation theory, non-commutative geometry, homological algebra, and
physics.

The notion of \ainf-category is a generalization of Stasheff's notion
of \ainf-algebra \cite{Stasheff:HomAssoc}. On the other hand,
\ainf-categories generalize differential graded categories. In contrast
to differential graded categories, composition in \ainf-categories is
associative only up to homotopy that satisfies certain equation up to
another homotopy, and so on. The notion of \ainf-category appeared in
the work of Fukaya on Floer homology \cite{Fukaya:A-infty} and was
related to mirror symmetry by Kontsevich
\cite{Kontsevich:alg-geom/9411018}. Basic concepts of the theory of
\ainf-categories have been developed by Fukaya
\cite{Fukaya:FloerMirror-II}, Keller \cite{math.RA/9910179},
Lef\`evre-Hasegawa \cite{Lefevre-Ainfty-these},
Lyubashenko~\cite{Lyu-AinfCat}, Soibelman \cite{MR2095670}.

The definition of \ainf-category does not assume the existence of
identity morphisms. The use of \ainf-categories without identities
requires caution: for example, there is no a sensible notion of
isomorphic objects, the notion of equivalence does not make sense, etc.
In order to develop a comprehensive theory of \ainf-categories, a
notion of unital \ainf-category, i.e., \ainf-category with identity
morphisms (also called units), is necessary. The obvious notion of
strictly unital \ainf-category, despite its technical advantages, is
not quite satisfactory: it is not homotopy invariant, meaning that it
does not translate along homotopy equivalences. Different definitions
of (weakly) unital \ainf-category have been suggested by the first
author \cite[Definition~7.3]{Lyu-AinfCat}, by Kontsevich and Soibelman
\cite[Definition~4.2.3]{math.RA/0606241}, and by Fukaya
\cite[Definition~5.11]{Fukaya:FloerMirror-II}. We prove that these
definitions are equivalent. The main ingredient of the proofs is the
Yoneda Lemma for unital (in the sense of Lyubashenko) \ainf-categories
proven in \cite[Appendix~A]{LyuMan-AmodSerre}, see also
\cite[Appendix~A]{math.CT/0306018}.

\section{Preliminaries}

We follow the notation and conventions of \cite{Lyu-AinfCat}, sometimes
without explicit mentioning. Some of the conventions are recalled here.

Throughout, \(\kk\) is a commutative ground ring. A graded
\(\kk\)\n-module always means a \(\ZZ\)\n-graded \(\kk\)\n-module.

A \emph{graded quiver} \(\ca\) consists of a set \(\Ob\ca\) of objects
and a graded \(\kk\)\n-module \(\ca(X,Y)\), for each \(X,Y\in\Ob\ca\).
A \emph{morphism of graded quivers} \(f:\ca\to\cb\) of degree \(n\)
consists of a function \(\Ob f:\Ob\ca\to\Ob\cb\), \(X\mapsto Xf\), and
a \(\kk\)\n-linear map \(f=f_{X,Y}:\ca(X,Y)\to\cb(Xf,Yf)\) of degree
\(n\), for each \(X,Y\in\Ob\ca\).

For a set \(S\), there is a category \(\cQuiver/S\) defined as follows.
Its objects are graded quivers whose set of objects is \(S\). A
morphism \(f:\ca\to\cb\) in \(\cQuiver/S\) is a morphism of graded
quivers of degree \(0\) such that \(\Ob f=\id_S\). The category
\(\cQuiver/S\) is monoidal. The tensor product of graded quivers
\(\ca\) and \(\cb\) is a graded quiver \(\ca\tens\cb\) such that
\[
(\ca\tens\cb)(X,Z)=\bigoplus_{Y\in S}\ca(X,Y)\tens\cb(Y,Z), \quad
X,Z\in S.
\]
The unit object is the \emph{discrete quiver} \(\kk S\) with \(\Ob\kk S=S\)
and
\[
(\kk S)(X,Y)=
\begin{cases}
\kk & \textup{ if \(X=Y\)},\\
0 & \textup{ if \(X\ne Y\)},
\end{cases}
\quad X,Y\in S.
\]
Note that a map of sets \(f:S\to R\) gives rise to a morphism of graded
quivers \(\kk f:\kk S\to\kk R\) with \(\Ob\kk f=f\) and \((\kk
f)_{X,Y}=\id_\kk\) is \(X=Y\) and \((\kk f)_{X,Y}=0\) if \(X\ne Y\),
\(X,Y\in S\).

An \emph{augmented graded cocategory} is a graded quiver \(\cc\)
equip\-ped with the structure of on augmented counital coassociative
coalgebra in the monoidal category \(\cQuiver/\Ob\cc\). Thus, \(\cc\)
comes with a comultiplication \(\Delta:\cc\to\cc\tens\cc\), a counit
\(\eps:\cc\to\kk\Ob\cc\), and an augmentation \(\eta:\kk\Ob\cc\to\cc\),
which are morphisms in \(\cQuiver/\Ob\cc\) satisfying the usual axioms.
A \emph{morphism of augmented graded cocategories} \(f:\cc\to\cd\) is a
morphism of graded quivers of degree \(0\) that preserves the
comultiplication, counit, and augmentation.

The main example of an augmented graded cocategory is the following.
Let \(\ca\) be a graded quiver. Denote by \(T\ca\) the direct sum of
graded quivers \(T^n\ca\), where \(T^n\ca=\ca^{\tens n}\) is the
\(n\)\n-fold tensor product of \(\ca\) in \(\cQuiver/\Ob\ca\); in
particular, \(T^0\ca=\kk\Ob\ca\), \(T^1\ca=\ca\),
\(T^2\ca=\ca\tens\ca\), etc. The graded quiver \(T\ca\) is an augmented
graded cocategory in which the comultiplication is the so called `cut'
comultiplication \(\Delta_0:T\ca\to T\ca\tens T\ca\) given by
\[
f_1\tens\dots\tens f_n\mapsto\sum_{k=0}^n f_1\tens\dots\tens
f_k\bigotimes f_{k+1}\tens\dots\tens f_n,
\]
the counit is given by the projection \(\pr_0:T\ca\to
T^0\ca=\kk\Ob\ca\), and the augmentation is given by the inclusion
\(\inj_0:\kk\Ob\ca=T^0\ca\hookrightarrow T\ca\).

The graded quiver \(T\ca\) admits also the structure of a graded
category, i.e., the structure of a unital associative algebra in the
monoidal category \(\cQuiver/\Ob\ca\). The multiplication
\(\mu:T\ca\tens T\ca\to T\ca\) removes brackets in tensors of the form
\((f_1\tdt f_m)\bigotimes(g_1\tdt g_n)\). The unit \(\eta:\kk\Ob\ca\to
T\ca\) is given by the inclusion
\(\inj_0:\kk\Ob\ca=T^0\ca\hookrightarrow T\ca\).

For a graded quiver \(\ca\), denote by \(s\ca\) its \emph{suspension},
the graded quiver given by \(\Ob s\ca=\Ob\ca\) and
\((s\ca(X,Y))^n=\ca(X,Y)^{n+1}\), for each \(n\in\ZZ\) and
\(X,Y\in\Ob\ca\). An \emph{\ainf-category} is a graded quiver \(\ca\) equipped
with a differential \(b:Ts\ca\to Ts\ca\) of degree \(1\) such that
\((Ts\ca,\Delta_0,\pr_0,\inj_0,b)\) is an \emph{augmented differential
graded cocategory}. In other terms, the equations
\[
b^2=0,\quad b\Delta_0=\Delta_0(b\tens1+1\tens b),\quad b\pr_0=0, \quad
\inj_0 b=0
\]
hold true. Denote by
\[
b_{mn}\defeq\bigl[
T^ms\ca\rto{\inj_m}Ts\ca\rto{b}Ts\ca\rto{\pr_n}T^ns\ca
\bigr]
\]
\emph{matrix coefficients} of \(b\), for \(m,n\ge0\). Matrix
coefficients \(b_{m1}\) are called \emph{components} of \(b\) and
abbreviated by \(b_m\). The above equations imply that \(b_0=0\) and
that \(b\) is unambiguously determined by its components via the
formula
\[
b_{mn}=\sum_{\substack{p+k+q=m\\ p+1+q=n}}1^{\tens p}\tens b_k\tens
1^{\tens q}:T^ms\ca\to T^ns\ca,\quad m,n\ge0.
\]
The equation \(b^2=0\) is equivalent to the system of equations
\[
\sum_{p+k+q=m}(1^{\tens p}\tens b_k\tens 1^{\tens
q})b_{p+1+q}=0:T^ms\ca\to s\ca,\quad m\ge1.
\]
For \ainf-categories \(\ca\) and \(\cb\), an \emph{\ainf-functor}
\(f:\ca\to\cb\) is a morphism of augmented differential graded
cocategories \(f:Ts\ca\to Ts\cb\). In other terms, \(f\) is a morphism
of augmented graded cocategories and preserves the differential,
meaning that \(fb=bf\). Denote by
\[
f_{mn}\defeq\bigl[
T^ms\ca\rto{\inj_m}Ts\ca\rto{f}Ts\cb\rto{\pr_n}T^ns\cb
\bigr]
\]
\emph{matrix coefficients} of \(f\), for \(m,n\ge0\). Matrix
coefficients \(f_{m1}\) are called \emph{components} of \(f\) and
abbreviated by \(f_m\). The condition that \(f\) is a morphism of
augmented graded cocategories implies that \(f_0=0\) and that \(f\) is
unambiguously determined by its components via the formula
\[
f_{mn}=\sum_{i_1+\dots+i_n=m}f_{i_1}\tdt f_{i_n}:T^ms\ca\to
T^ns\cb,\quad m,n\ge0.
\]
The equation \(fb=bf\) is equivalent to the system of equations
    \ifx\chooseClass1
\begin{multline*}
\sum_{i_1+\dots+i_n=m}(f_{i_1}\tdt f_{i_n})b_n
\\
=\sum_{p+k+q=m}(1^{\tens p}\tens b_k\tens 1^{\tens q})
f_{p+1+q}:T^ms\ca\to s\cb,
\end{multline*}
    \else
\[ \sum_{i_1+\dots+i_n=m}(f_{i_1}\tdt f_{i_n})b_n
=\sum_{p+k+q=m}(1^{\tens p}\tens b_k\tens 1^{\tens q})
f_{p+1+q}:T^ms\ca\to s\cb,
\]
    \fi
for \(m\ge1\). An \ainf-functor \(f\) is called \emph{strict} if
\(f_n=0\) for \(n>1\).

\section{Definitions}

\begin{definition}[cf. \cite{Fukaya:FloerMirror-II,math.RA/9910179}]
An \ainf-category $\ca$ is \emph{strictly unital} if, for each
$X\in\Ob\ca$, there is a $\kk$-linear map
\(\sS{_X}\uni^\ca_0:\kk\to(s\ca)^{-1}(X,X)\), called a \emph{strict
unit}, such that the following conditions are satisfied:
$\sS{_X}\uni^\ca_0b_1=0$, the chain maps
$(1\tens\sS{_Y}\uni^\ca_0)b_2,-(\sS{_X}\uni^\ca_0\tens1)b_2:
 s\ca(X,Y)\to s\ca(X,Y)$ are equal to the identity map, for each
\(X,Y\in\Ob\ca\), and $(\cdots\tens\uni^\ca_0\tens\cdots)b_n=0$ if
$n\ge3$.
\end{definition}

For example, differential graded categories are strictly unital.

\begin{definition}[Lyubashenko {\cite[Definition~7.3]{Lyu-AinfCat}}]
An \ainf-ca\-te\-go\-ry \(\ca\) is \emph{unital} if, for each $X\in\Ob\ca$,
there is a $\kk$-linear map $\sS{_X}\uni^\ca_0:\kk\to(s\ca)^{-1}(X,X)$,
called a \emph{unit}, such that the following conditions hold:
$\sS{_X}\uni^\ca_0b_1=0$ and the chain maps
$(1\tens\sS{_Y}\uni^\ca_0)b_2,-(\sS{_X}\uni^\ca_0\tens1)b_2:
 s\ca(X,Y)\to s\ca(X,Y)$ are homotopic to the identity map, for each
\(X,Y\in\Ob\ca\). An arbitrary homotopy between
\((1\tens\sS{_Y}\uni^\ca_0)b_2\) and the identity map is called a
\emph{right unit homotopy}. Similarly, an arbitrary homotopy between
\(-(\sS{_X}\uni^\ca_0\tens1)b_2\) and the identity map is called a
\emph{left unit homotopy}. An \ainf-functor \(f:\ca\to\cb\) between
unital \ainf-categories is \emph{unital} if the cycles
\(\sS{_X}\uni^\ca_0f_1\) and \(\sS{_{Xf}}\uni^\cb_0\) are cohomologous,
i.e., differ by a boundary, for each \(X\in\Ob\ca\).
\end{definition}

Clearly, a strictly unital \ainf-category is unital.

With an arbitrary \ainf-category $\ca$ a strictly unital \ainf-category
$\ca^\su$ with the same set of objects is associated. For each
$X,Y\in\Ob\ca$, the graded $\kk$-module $s\ca^\su(X,Y)$ is given by
\[ s\ca^\su(X,Y)=
\begin{cases}
s\ca(X,Y) & \textup{ if \(X\neq Y\)},\\
s\ca(X,X)\oplus\kk\sS{_X}\uni^{\ca^\su}_0& \textup{ if \(X=Y\)},
\end{cases}
\]
where $\sS{_X}\uni^{\ca^\su}_0$ is a new generator of degree $-1$. The
element $\sS{_X}\uni^{\ca^\su}_0$ is a strict unit by definition, and
the natural embedding $e:\ca\hookrightarrow\ca^\su$ is a strict
\ainf-functor.

\begin{definition}[Kontsevich--Soibelman {\cite[Definition~4.2.3]{math.RA/0606241}}]
A \emph{weak unit} of an \ainf-category $\ca$ is an \ainf-functor
$U:\ca^\su\to\ca$ such that
\[
\bigl[\ca \xhookrightarrow{e} \ca^\su \rto{U} \ca\bigr] =\id_\ca.
\]
\end{definition}

\begin{proposition}\label{pro-unital-structure-implies-unitality}
Suppose that an \ainf-category $\ca$ admits a weak unit. Then the
\ainf-category $\ca$ is unital.
\end{proposition}

\begin{proof}
Let $U:\ca^\su\to\ca$ be a weak unit of $\ca$. For each \(X\in\Ob\ca\),
denote by \(\sS{_X}\uni^\ca_0\) the element
\(\sS{_X}\uni^{\ca^\su}_0U_1\in s\ca(X,X)\) of degree \(-1\). It
follows from the equation \(U_1b_1=b_1U_1\) that
\(\sS{_X}\uni^\ca_0b_1=0\). Let us prove that \(\sS{_X}\uni^\ca_0\) are
unit elements of $\ca$.

For each $X,Y\in\Ob\ca$, there is a \(\kk\)\n-linear map
\[
h=(1\tens\sS{_Y}\uni_0)U_2:s\ca(X,Y)\to s\ca(X,Y)
\]
of degree $-1$. The equation
\begin{equation}\label{eq-b1u2-u2b1}
(1\tens b_1+b_1\tens1)U_2 + b_2U_1 = U_2b_1 + (U_1\tens U_1)b_2
\end{equation}
implies that
\[ -b_1h+1 = hb_1 +(1\tens\sS{_Y}\uni^\ca_0)b_2:s\ca(X,Y)\to s\ca(X,Y),
\]
thus $h$ is a right unit homotopy for $\ca$. For each $X,Y\in\Ob\ca$,
there is a \(\kk\)\n-linear map
\[
h'=-(\sS{_X}\uni_0\tens1)U_2:s\ca(X,Y)\to s\ca(X,Y)
\]
of degree $-1$. Equation~\eqref{eq-b1u2-u2b1} implies that
\[ b_1h'-1=-h'b_1 +(\sS{_X}\uni^\ca_0\tens1)b_2:s\ca(X,Y)\to s\ca(X,Y),
\]
thus $h'$ is a left unit homotopy for $\ca$. Therefore, $\ca$ is
unital.
\end{proof}

\begin{definition}[Fukaya {\cite[Definition~5.11]{Fukaya:FloerMirror-II}}]
\label{def-homotopy-unital-ainf-category}
An \ainf-category $\cc$ is called \emph{homotopy unital} if the graded
quiver
\[ \cc^+ = \cc\oplus\kk\cc\oplus s\kk\cc
\]
(with \(\Ob\cc^+=\Ob\cc\)) admits an \ainf-structure $b^+$ of the
following kind. Denote the generators of the second and the third
direct summands of the graded quiver \(s\cc^+=s\cc\oplus s\kk\cc\oplus
s^2\kk\cc\) by \(\sS{_X}\uni^{\cc^\su}_0=1s\) and \(\bj^\cc_X=1s^2\) of
degree respectively $-1$ and $-2$, for each \(X\in\Ob\cc\). The
conditions on $b^+$ are:
\begin{enumerate}
\renewcommand{\labelenumi}{(\arabic{enumi})}
\item for each \(X\in\Ob\cc\), the element
\(\sS{_X}\uni^\cc_0\defeq\sS{_X}\uni^{\cc^\su}_0-\bj^\cc_Xb^+_1\) is
contained in \(s\cc(X,X)\);

\item $\cc^+$ is a strictly unital \ainf-category with strict
units \(\sS{_X}\uni^{\cc^\su}_0\), \(X\in\Ob\cc\);

\item the embedding \(\cc\hookrightarrow\cc^+\) is a strict
\ainf-functor;

\item \((s\cc\oplus s^2\kk\cc)^{\tens n}b^+_n\subset s\cc\), for each
$n>1$.
\end{enumerate}
\end{definition}

In particular, $\cc^+$ contains the strictly unital \ainf-category
\(\cc^\su=\cc\oplus\kk\cc\). A version of this definition suitable for
filtered \ainf-algebras (and filtered \ainf-categories) is given by
Fukaya, Oh, Ohta and Ono in their book
\cite[Definition~8.2]{FukayaOhOhtaOno:Anomaly}.

Let $\cd$ be a strictly unital \ainf-category with strict units
\(\uni^\cd_0\). Then it has a canonical homotopy unital structure
\((\cd^+,b^+)\). Namely,
\(\bj^\cd_Xb^+_1=\sS{_X}\uni^{\cd^\su}_0-\sS{_X}\uni^\cd_0\), and
\(b^+_n\) vanishes for each $n>1$ on each summand of
 \((s\cd\oplus s^2\kk\cd)^{\tens n}\) except on \(s\cd^{\tens n}\),
where it coincides with \(b^\cd_n\). Verification of the equation
\((b^+)^2=0\) is a straightforward computation.

\begin{proposition}\label{prop-homotopy-unital-unital}
An arbitrary homotopy unital \ainf-category is unital.
\end{proposition}

\begin{proof}
Let \(\cc\subset\cc^+\) be a homotopy unital category. We claim that
the distinguished cycles \(\sS{_X}\uni^\cc_0\in\cc(X,X)[1]^{-1}\),
\(X\in\Ob\cc\), turn \(\cc\) into a unital \ainf-category. Indeed, the
identity
\[
(1\tens b^+_1+b^+_1\tens1)b^+_2+b^+_2b^+_1=0
\]
applied to \(s\cc\tens\bj^\cc\) or to \(\bj^\cc\tens s\cc\) implies
\begin{alignat*}2
(1\tens\uni^\cc_0)b^\cc_2 &= 1 +(1\tens\bj^\cc)b^+_2b^\cc_1
+b^\cc_1(1\tens\bj^\cc)b^+_2 &&: s\cc \to s\cc,
\\
(\uni^\cc_0\tens1)b^\cc_2 &= -1 +(\bj^\cc\tens1)b^+_2b^\cc_1
+b^\cc_1(\bj^\cc\tens1)b^+_2 &&: s\cc \to s\cc.
\end{alignat*}
Thus, \((1\tens\bj^\cc)b^+_2:s\cc \to s\cc\) and
\((\bj^\cc\tens1)b^+_2:s\cc \to s\cc\) are unit homotopies. Therefore,
the \ainf-category \(\cc\) is unital.
\end{proof}

The converse of \propref{prop-homotopy-unital-unital} holds true as
well.

\begin{theorem}\label{thm-homotopy-unital-category-unital}
An arbitrary unital \ainf-category $\cc$ with unit elements
\(\uni^\cc_0\) admits a homotopy unital structure \((\cc^+,b^+)\) with
 \(\bj^\cc b^+_1=\uni^{\cc^\su}_0-\uni^\cc_0\).
\end{theorem}

\begin{proof}
By \cite[Corollary~A.12]{math.CT/0306018}, there exists a differential
graded category \(\cd\) and an \ainf-equivalence \(\phi:\cc\to\cd\). By
\cite[Remark~A.13]{math.CT/0306018}, we may choose \(\cd\) and \(\phi\)
such that \(\Ob\cd=\Ob\cc\) and \(\Ob\phi=\id_{\Ob\cc}\). Being
strictly unital $\cd$ admits a canonical homotopy unital structure
\((\cd^+,b^+)\). In the sequel, we may assume that $\cd$ is a strictly
unital \ainf-category equivalent to $\cc$ via $\phi$ with the mentioned
properties. Let us construct simultaneously an \ainf-structure $b^+$ on
$\cc^+$ and an \ainf-functor \(\phi^+:\cc^+\to\cd^+\) that will turn
out to be an equivalence.

Let us extend the homotopy isomorphism \(\phi_1:s\cc\to s\cd\) to a
chain quiver map \(\phi_1^+:s\cc^+\to s\cd^+\). The \ainf-equivalence
\(\phi:\cc\to\cd\) is a unital \ainf-functor, i.e., for each
\(X\in\Ob\cc\), there exists \(v_X\in\cd(X,X)[1]^{-2}\) such that
\(\sS{_X}\uni^\cd_0 -\sS{_X}\uni^\cc_0\phi_1 =v_Xb_1\). In order that
\(\phi^+\) be strictly unital, we define
\(\sS{_X}\uni^{\cc^\su}_0\phi^+_1=\sS{_X}\uni^{\cd^\su}_0\). We should
have
\begin{multline*}
\bj^\cc_X\phi^+_1b^+_1 =\bj^\cc_Xb^+_1\phi^+_1
=\sS{_X}\uni^{\cc^\su}_0\phi^+_1-\sS{_X}\uni^\cc_0\phi_1
    \ifx\chooseClass1
\\
    \fi
=\sS{_X}\uni^{\cd^\su}_0 -\sS{_X}\uni^\cd_0 +\sS{_X}\uni^\cd_0
-\sS{_X}\uni^\cc_0\phi_1 =(\bj^\cc_X+v_X)b^+_1,
\end{multline*}
so we define \(\bj^\cc_X\phi^+_1=\bj^\cd_X+v_X\).

We claim that there is a homotopy unital structure \((\cc^+,b^+)\) of
$\cc$ satisfying the four conditions of
\defref{def-homotopy-unital-ainf-category}
and an \ainf-functor \(\phi^+:\cc^+\to\cd^+\) satisfying four parallel
conditions:
\begin{enumerate}
\renewcommand{\labelenumi}{(\arabic{enumi})}
\item the first component of $\phi^+$ is the quiver morphism $\phi^+_1$
constructed above;

\item the \ainf-functor $\phi^+$ is strictly unital;

\item the restriction of $\phi^+$ to \(\cc\) gives $\phi$;

\item \((s\cc\oplus s^2\kk\cc)^{\tens n}\phi^+_n\subset s\cd\), for
each $n>1$.
\end{enumerate}
Notice that in the presence of conditions (2) and (3) the first
condition reduces to \(\bj^\cc_X(\phi^+)_1=\bj^\cd_X+v_X\), for each
$X\in\Ob\cc$.

Components of the (1,1)\n-coderivation \(b^+:Ts\cc^+\to Ts\cc^+\) of
degree 1 and of the augmented graded cocategory morphism
\(\phi^+:Ts\cc^+\to Ts\cd^+\) are constructed by induction. We already
know components \(b^+_1\) and \(\phi^+_1\). Given an integer $n\ge2$,
assume that we have already found components $b^+_m$, $\phi^+_m$ of the
sought $b^+$ and $\phi^+$ for $m<n$ such that the equations
\begin{alignat}2
((b^+)^2)_m &= 0 &: T^ms\cc^+(X,Y) &\to s\cc^+(X,Y),
\label{eq-b2m-0-C-C}
\\
(\phi^+b^+)_m &= (b^+\phi^+)_m &: T^ms\cc^+(X,Y) &\to s\cd^+(Xf,Yf)
\label{eq-fbm-bfm-0-C-B}
\end{alignat}
are satisfied for all $m<n$. Define $b^+_n$, $\phi^+_n$ on direct
summands of \(T^ns\cc^+\) which contain a factor \(\uni^{\cc^\su}_0\)
by the requirement of strict unitality of $\cc^+$ and $\phi^+$. Then
equations \eqref{eq-b2m-0-C-C}, \eqref{eq-fbm-bfm-0-C-B} hold true for
$m=n$ on such summands. Define $b^+_n$, $\phi^+_n$ on the direct
summand \(T^ns\cc\subset T^ns\cc^+\) as $b^\cc_n$ and $\phi_n$. Then
equations \eqref{eq-b2m-0-C-C}, \eqref{eq-fbm-bfm-0-C-B} hold true for
$m=n$ on the summand \(T^ns\cc\). It remains to construct those
components of \(b^+\) and \(\phi^+\) which have \(\bj^\cc\) as one of
their arguments.

Extend $b_1:s\cc\to s\cc$ to \(b'_1:s\cc^+\to s\cc^+\) by
\(\uni^{\cc^\su}_0b'_1=0\) and \(\bj^\cc b'_1=0\). Define
\(b^-_1=b^+_1-b'_1:s\cc^+\to s\cc^+\). Thus,
\(b^-_1\big|_{s\cc^\su}=0\),
 \(\bj^\cc b^-_1=\uni^{\cc^\su}_0-\uni^\cc_0\) and \(b^+_1=b'_1+b^-_1\).
Introduce for \(0\le k\le n\) the graded subquiver \(\cn_k\subset T^n(s\cc\oplus
s^2\kk\cc)\) by
\[
\cn_k=
\bigoplus_{p_0+p_1+\dots+p_k+k=n}
 T^{p_0}s\cc\tens\bj^\cc\tens T^{p_1}s\cc\tdt\bj^\cc\tens T^{p_k}s\cc
\]
stable under the differential
 \(d^{\cn_k}=\sum_{p+1+q=n}1^{\tens p}\tens b'_1\tens1^{\tens q}\), and
the graded subquiver \(\cp_l\subset T^ns\cc^+\) by
\[
\cp_l=
\bigoplus_{p_0+p_1+\dots+p_l+l=n}
 T^{p_0}s\cc^\su\tens\bj^\cc\tens T^{p_1}s\cc^\su\tdt\bj^\cc\tens T^{p_l}s\cc^\su.
\]
There is also the subquiver
\[
\cq_k = \bigoplus_{l=0}^k \cp_l \subset T^ns\cc^+
\]
and its complement
\[
\cq_k^\bot = \bigoplus_{l=k+1}^n \cp_l \subset T^ns\cc^+.
\]
Notice that the subquiver $\cq_k$ is stable under the differential
 \(d^{\cq_k}=\sum_{p+1+q=n}1^{\tens p}\tens b^+_1\tens1^{\tens q}\),
and $\cq_k^\bot$ is stable under the differential
 \(d^{\cq^\bot_k}=\sum_{p+1+q=n}1^{\tens p}\tens b'_1\tens1^{\tens q}\).
Furthermore, the image of
 \(1^{\tens a}\tens b^-_1\tens1^{\tens c}:\cn_k\to T^ns\cc^+\) is
contained in \(\cq_{k-1}\) for all $a,c\ge0$ such that \(a+1+c=n\).

Firstly, the components $b^+_n$, $\phi^+_n$ are defined on the
graded subquivers \(\cn_0=T^ns\cc\) and \(\cq_0=T^ns\cc^\su\). Assume
for an integer \(0<k\le n\) that restrictions of $b^+_n$, $\phi^+_n$ to
\(\cn_l\) are already found for all $l<k$. In other terms, we are given
\(b^+_n:\cq_{k-1}\to s\cc^+\), \(\phi^+_n:\cq_{k-1}\to s\cd\) such that
equations \eqref{eq-b2m-0-C-C}, \eqref{eq-fbm-bfm-0-C-B} hold on
\(\cq_{k-1}\). Let us construct the restrictions $b^+_n:\cn_k\to s\cc$,
$\phi^+_n:\cn_k\to s\cd$, performing the induction step.

Introduce a (1,1)\n-coderivation $\tilde{b}:Ts\cc^+\to Ts\cc^+$ of
degree $1$ by its components
 $(0,b^+_1,\dots,b^+_{n-1}, \pr_{\cq_{k-1}}\cdot
b^+_n|_{\cq_{k-1}},0,\dots)$. Introduce also a morphism of augmented
graded cocategories
 \(\tilde{\phi}:Ts\cc^+\to Ts\cd^+\) with \(\Ob\tilde{\phi}=\Ob\phi\)
by its components
    \ifx\chooseClass1
$(\phi^+_1,\dots,\phi^+_{n-1},\pr_{\cq_{k-1}}\cdot\phi^+_n|_{\cq_{k-1}},0,\dots)$.
    \else
\[
(\phi^+_1,\dots,\phi^+_{n-1},\pr_{\cq_{k-1}}\cdot\phi^+_n|_{\cq_{k-1}},0,\dots).
\]
    \fi
Here
\(\pr_{\cq_{k-1}}:T^ns\cc^+\to\cq_{k-1}\) is the natural projection,
vanishing on $\cq_{k-1}^\bot$. Then
\(\lambda\defeq\tilde{b}^2:Ts\cc^+\to Ts\cc^+\) is a
(1,1)\n-coderivation of degree $2$ and
 \(\nu\defeq-\tilde{\phi}b^++\tilde{b}\tilde{\phi}:
 Ts\cc^+\to Ts\cd^+\)
is a $(\tilde{\phi},\tilde{\phi})$-coderivation of degree $1$.
Equations \eqref{eq-b2m-0-C-C}, \eqref{eq-fbm-bfm-0-C-B} imply that
$\lambda_m=0$, $\nu_m=0$ for $m<n$. Moreover, $\lambda_n$, $\nu_n$
vanish on \(\cq_{k-1}\). On the complement the $n$\n-th components equal
    \ifx\chooseClass1
\begin{align*}
\lambda_n =& \sum_{a+r+c=n}^{1<r<n}
(1^{\tens a}\tens b^+_r\tens1^{\tens c})b^+_{a+1+c}
\\
+&\sum_{a+1+c=n}(1^{\tens a}\tens b^-_1\tens1^{\tens c})\tilde{b}_n:
\cq_{k-1}^\bot\to s\cc^+,
\\
\nu_n =&
-\sum_{i_1+\dots+i_r=n}^{1<r\le n} (\phi^+_{i_1}\tdt\phi^+_{i_r})b^+_r
\\
+&\sum_{a+r+c=n}^{1<r<n}
(1^{\tens a}\tens b^+_r\tens1^{\tens c})\phi^+_{a+1+c}
\\
+&\sum_{a+1+c=n}(1^{\tens a}\tens b^-_1\tens1^{\tens c})\tilde{\phi}_n:
\cq_{k-1}^\bot\to s\cd.
\end{align*}
    \else
\begin{align*}
\lambda_n =& \sum_{a+r+c=n}^{1<r<n}
(1^{\tens a}\tens b^+_r\tens1^{\tens c})b^+_{a+1+c}
+\sum_{a+1+c=n}(1^{\tens a}\tens b^-_1\tens1^{\tens c})\tilde{b}_n:
\cq_{k-1}^\bot\to s\cc^+,
\\
\nu_n =&
-\sum_{i_1+\dots+i_r=n}^{1<r\le n} (\phi^+_{i_1}\tdt\phi^+_{i_r})b^+_r
+\sum_{a+r+c=n}^{1<r<n}
(1^{\tens a}\tens b^+_r\tens1^{\tens c})\phi^+_{a+1+c}
\\
 &\hphantom{\sum_{a+r+c=n}^{1<r<n}(1^{\tens a}\tens b^+_r\tens1^{\tens c})b^+_{a+1+c}}
+\sum_{a+1+c=n}(1^{\tens a}\tens b^-_1\tens1^{\tens c})\tilde{\phi}_n:
\cq_{k-1}^\bot\to s\cd.
\end{align*}
    \fi
The restriction \(\lambda_n|_{\cn_k}\) takes values in $s\cc$. Indeed,
for the first sum in the expression for $\lambda_n$ this follows by the
induction assumption since $r>1$ and \(a+1+c>1\). For the second sum
this follows by the induction assumption and strict unitality if $n>2$.
In the case of $n=2$, $k=1$ this is also straightforward. The only case
which requires computation is $n=2$, $k=2$:
\[ (\bj^\cc\tens\bj^\cc)(1\tens b^-_1+b^-_1\tens1)\tilde{b}_2
=\bj^\cc -(\bj^\cc\tens\uni^\cc_0)b^+_2 -\bj^\cc
-(\uni^\cc_0\tens\bj^\cc)b^+_2,
\]
which belongs to $s\cc$ by the induction assumption.

Equations \eqref{eq-b2m-0-C-C}, \eqref{eq-fbm-bfm-0-C-B} for $m=n$ take
the form
\begin{align}
-b^+_nb_1 -\sum_{a+1+c=n}(1^{\tens a}\tens b'_1\tens1^{\tens c})b^+_n
=\lambda_n: \cn_k &\to s\cc,
 \label{eq-bnb1-lambda-TsC-sC}
\\
\phi^+_nb_1
-\sum_{a+1+c=n}(1^{\tens a}\tens b'_1\tens1^{\tens c})\phi^+_n
-b^+_n\phi_1 =\nu_n: \cn_k &\to s\cd.
 \label{eq-fnb1-nu-TsC-sB}
\end{align}
For arbitrary objects $X$, $Y$ of $\cc$, equip the graded
\(\kk\)\n-module \(\cn_k(X,Y)\) with the differential
\(d^{\cn_k}=\sum_{p+1+q=n}1^{\tens p}\tens b'_1\tens1^{\tens q}\) and
denote by \(u\) the chain map
\begin{align*}
\uCom(\cn_k(X,Y),s\cc(X,Y)) &\to \uCom(\cn_k(X,Y),s\cd(X\phi,Y\phi)),\\
\lambda&\mapsto\lambda\phi_1.
\end{align*}
Since $\phi_1$ is homotopy invertible, the map $u$ is homotopy
invertible as well. Therefore, the complex $\Cone(u)$ is contractible,
e.g. by \cite[Lemma~B.1]{Lyu-AinfCat}, in particular, acyclic.
Equations \eqref{eq-bnb1-lambda-TsC-sC} and \eqref{eq-fnb1-nu-TsC-sB}
have the form $-b^+_nd=\lambda_n$, $\phi^+_nd+b^+_nu=\nu_n$, that is,
the element \((\lambda_n,\nu_n)\) of
    \ifx\chooseClass1
\begin{multline*}
\ucom^2(\cn_k(X,Y),s\cc(X,Y))\oplus\ucom^1(\cn_k(X,Y),s\cd(X\phi,Y\phi))
\\
=\Cone^1(u)
\end{multline*}
    \else
\[ \ucom^2(\cn_k(X,Y),s\cc(X,Y))\oplus\ucom^1(\cn_k(X,Y),s\cd(X\phi,Y\phi))
=\Cone^1(u)
\]
    \fi
has to be the boundary of the sought element \((b^+_n,\phi^+_n)\) of
    \ifx\chooseClass1
\begin{multline*}
\ucom^1(\cn_k(X,Y),s\cc(X,Y))\oplus\ucom^0(\cn_k(X,Y),s\cd(X\phi,Y\phi))
\\
=\Cone^0(u).
\end{multline*}
    \else
\[ \ucom^1(\cn_k(X,Y),s\cc(X,Y))\oplus\ucom^0(\cn_k(X,Y),s\cd(X\phi,Y\phi))
=\Cone^0(u).
\]
    \fi
These equations are solvable because \((\lambda_n,\nu_n)\) is a cycle
in \(\Cone^1(u)\). Indeed, the equations to verify $-\lambda_nd=0$,
$\nu_nd+\lambda_nu=0$ take the form
\begin{align*}
-\lambda_nb_1
+\sum_{p+1+q=n}(1^{\tens p}\tens b'_1\tens1^{\tens q})\lambda_n =0:
\cn_k &\to s\cc,
\\
\nu_nb_1 +\sum_{p+1+q=n}(1^{\tens p}\tens b'_1\tens1^{\tens q})\nu_n
-\lambda_n\phi_1 =0: \cn_k &\to s\cd.
\end{align*}
Composing the identity
\(-\lambda\tilde{b}+\tilde{b}\lambda=0:T^ns\cc^+\to Ts\cc^+\) with the
projection \(\pr_1:Ts\cc^+\to s\cc^+\) yields the first equation. The
second equation follows by composing the identity
 \(\nu b^++\tilde{b}\nu-\lambda\tilde{\phi}=0:T^ns\cc^+\to Ts\cd^+\)
with \(\pr_1:Ts\cd^+\to s\cd^+\).

Thus, the required restrictions of $b^+_n$, $\phi^+_n$ to $\cn_k$ (and
to $\cq_k$) exist and satisfy the required equations. We proceed by
induction increasing $k$ from $0$ to $n$ and determining $b^+_n$,
$\phi^+_n$ on the whole \(\cq_n=T^ns\cc^+\). Then we replace $n$ with
$n+1$ and start again from \(T^{n+1}s\cc\). Thus the induction on $n$
goes through.
\end{proof}

\begin{remark}
Let \((\cc^+,b^+)\) be a homotopy unital structure of an \ainf-category
$\cc$. Then the embedding \ainf-functor \(\iota:\cc\to\cc^+\) is an
equivalence. Indeed, it is bijective on objects. By
\cite[Theorem~8.8]{Lyu-AinfCat} it suffices to prove that
\(\iota_1:s\cc\to s\cc^+\) is homotopy invertible. And indeed, the
chain quiver map \(\pi_1:s\cc^+\to s\cc\), \(\pi_1|_{s\cc}=\id\),
\(\sS{_X}\uni^{\cc^\su}_0\pi_1=\sS{_X}\uni^\cc_0\),
\(\bj^\cc_X\pi_1=0\), is homotopy inverse to $\iota_1$. Namely, the
homotopy \(h:s\cc^+\to s\cc^+\), \(h|_{s\cc}=0\),
\(\sS{_X}\uni^{\cc^\su}_0h=\bj^\cc_X\), \(\bj^\cc_Xh=0\), satisfies the
equation \(\id_{s\cc^+}-\pi_1\cdot\iota_1=hb^+_1+b^+_1h\).

The equation between \ainf-functors
\[
\bigl[\cc \rto{\iota^\cc} \cc^+ \rto{\phi^+} \cd^+\bigr] =\bigl[\cc
\rto\phi \cd \rto{\iota^\cd} \cd^+\bigr]
\]
obtained in the proof of \thmref{thm-homotopy-unital-category-unital}
implies that \(\phi^+\) is an \ainf-equivalence as well. In particular,
\(\phi^+_1\) is homotopy invertible.
\end{remark}

The converse of \propref{pro-unital-structure-implies-unitality} holds
true as well, however its proof requires more preliminaries. It is
deferred until \secref{sec-proof-unital-implies-unital-structure}.

\section{Double coderivations}

\begin{definition}
For \ainf-functors \(f,g:\ca\to\cb\), a \emph{double
\((f,g)\)\n-coderivation} of degree \(d\) is a system of
\(\kk\)\n-linear maps
\[
r:(Ts\ca\tens Ts\ca)(X,Y)\to Ts\cb(Xf,Yg),\quad X,Y\in\Ob\ca,
\]
of degree \(d\) such that the equation
\begin{equation}\label{equ-rDelta-Delta1fr-1Deltarg}
r\Delta_0=(\Delta_0\tens1)(f\tens r)+(1\tens\Delta_0)(r\tens g)
\end{equation}
holds true.
\end{definition}

Equation~\eqref{equ-rDelta-Delta1fr-1Deltarg} implies that $r$ is
determined by a system of $\kk$-linear maps $r\pr_1:Ts\ca\tens Ts\ca\to
s\cb$ with components of degree $d$
    \ifx\chooseClass1
\begin{multline*}
r_{n,m}:s\ca(X_0,X_1)\tdt s\ca(X_{n+m-1},X_{n+m})
\\
\to s\cb(X_0f,X_{n+m}g),
\end{multline*}
    \else
\[ r_{n,m}:s\ca(X_0,X_1)\tdt s\ca(X_{n+m-1},X_{n+m})
\to s\cb(X_0f,X_{n+m}g),
\]
    \fi
for \(n,m\ge0\), via the formula
\begin{align}
r_{n,m;k}&=(r|_{T^ns\ca\tens T^ms\ca})\pr_k:T^ns\ca\tens T^ms\ca\to T^ks\cb,\nonumber\\
r_{n,m;k}&=
\sum_{\substack{i_1+\dots+i_p+i=n,\\ j_1+\dots+j_q+j=m}}^{p+1+q=k}
f_{i_1}\tdt f_{i_p}\tens r_{i,j}\tens g_{j_1}\tdt g_{j_q}.
\end{align}
This follows from the equation
    \ifx\chooseClass1
\begin{multline}\label{equ-r-Delta-iterated}
r\Delta_0^{(l)}=\sum_{p+1+q=l}(\Delta_0^{(p+1)}\tens\Delta_0^{(q+1)})
(f^{\tens p}\tens r\tens g^{\tens q}):
\\
Ts\ca\tens Ts\ca\to(Ts\cb)^{\tens l},
\end{multline}
    \else
\begin{equation}\label{equ-r-Delta-iterated}
r\Delta_0^{(l)}=\sum_{p+1+q=l}(\Delta_0^{(p+1)}\tens\Delta_0^{(q+1)})
(f^{\tens p}\tens r\tens g^{\tens q}):
Ts\ca\tens Ts\ca\to(Ts\cb)^{\tens l},
\end{equation}
    \fi
which holds true for each \(l\ge 0\). Here $\Delta_0^{(0)}=\eps$,
$\Delta_0^{(1)}=\id$, $\Delta_0^{(2)}=\Delta_0$ and $\Delta_0^{(l)}$
means the cut comultiplication iterated $l-1$ times.

Double \((f,g)\)\n-coderivations form a chain complex, which we are
going to denote by \((\McD(\ca,\cb)(f,g),B_1)\). For each \(d\in\ZZ\),
the component $\McD(\ca,\cb)(f,g)^d$ consists of double
\((f,g)\)\n-coderi\-vations of degree $d$. The differential \(B_1\) of
degree \(1\) is given by
\[
rB_1\defeq rb-(-)^d(1\tens b+b\tens1)r,
\]
for each \(r\in\McD(\ca,\cb)(f,g)^d\). The component \([rB_1]_{n,m}\) of $rB_1$
is given by
\begin{multline}\label{eq-diff}
\sum_{\substack{i_1+\dots+i_p+i=n,\\ j_1+\dots+j_q+j=m}}
(f_{i_1}\tdt f_{i_p}\tens r_{ij}\tens g_{j_1}\tdt g_{j_q})b_{p+1+q}
\\
-(-)^r\sum_{a+k+c=n}(1^{\tens a}\tens b_k\tens 1^{\tens c+m})r_{a+1+c,m}
\\
-(-)^r\sum_{u+t+v=m} (1^{\tens n+u}\tens b_t\tens 1^{\tens
v})r_{n,u+1+v},
\end{multline}
for each \(n,m\ge0\). An \ainf-functor \(h:\cb\to\cc\)  gives rise to a
chain map
\[
\McD(\ca,\cb)(f,g) \to \McD(\ca,\cc)(fh,gh),\quad r\mapsto rh.
\]
The component \([rh]_{n,m}\) of $rh$ is given by
\begin{equation}\label{eq-functor}
\sum_{\substack{i_1+\dots+i_p+i=n,\\ j_1+\dots+j_q+j=m}}
(f_{i_1}\tdt f_{i_p}\tens r_{i,j}\tens g_{j_1}\tdt g_{j_q})h_{p+1+q},
\end{equation}
for each \(n,m\ge0\). Similarly, an \ainf-functor \(k:\cd\to\ca\) gives
rise to a chain map
\[
\McD(\ca,\cb)(f,g)\to\McD(\cd,\cb)(kf,kg),\quad r\mapsto (k\tens k)r.
\]
The component \([(k\tens k)r]_{n,m}\) of \((k\tens k)r\) is given by
\begin{equation}\label{equ-comp-psi-psi-r}
\sum_{\substack{i_1+\dots+i_p=n\\ j_1+\dots+j_q=m}}
(k_{i_1}\tdt k_{i_p}\tens k_{j_1}\tdt k_{j_q})r_{p,q},
\end{equation}
for each \(n,m\ge0\). Proofs of these facts are elementary and are left
to the reader.

Let \(\cc\) be an \ainf-category. For each \(n\ge0\), introduce a
morphism
\[ \nu_n
=\sum_{i=0}^n(-)^{n-i}(1^{\tens i}\tens\eps\tens 1^{\tens n-i}):
(Ts\cc)^{\tens n+1}\to(Ts\cc)^{\tens n},
\]
in \(\cQuiver/\Ob\cc\). In particular,
\(\nu_0=\eps:Ts\cc\to\kk\Ob\cc\). Denote
\(\nu=\nu_1=(1\tens\eps)-(\eps\tens 1): Ts\cc\tens Ts\cc\to Ts\cc\) for
the sake of brevity.

\begin{lemma}\label{lem-nu-double-coder}
The map \(\nu: Ts\cc\tens Ts\cc\to Ts\cc\) is a double
\((1,1)\)\n-coderivation of degree~0 and \(\nu B_1=0\).
\end{lemma}

\begin{proof}
We have:
\begin{multline*}
(\Delta_0\tens 1)(1\tens\nu)+(1\tens\Delta_0)(\nu\tens 1)
\\
\qquad
=(\Delta_0\tens 1)(1\tens 1\tens\eps)
-(\Delta_0\tens1)(1\tens\eps\tens 1)
\hfill
\\
\hfill
+(1\tens\Delta_0)(1\tens\eps\tens 1)
-(1\tens\Delta_0)(\eps\tens 1\tens 1)
\quad
\\
\qquad =(\Delta_0\tens\eps)-(\eps\tens\Delta_0)
=((1\tens\eps)-(\eps\tens 1))\Delta_0=\nu\Delta_0, \hfill
\end{multline*}
due to the identities
    \ifx\chooseClass1
\begin{multline*}
(\Delta_0\tens 1)(1\tens\eps\tens 1)=1\tens1
=(1\tens\Delta_0)(1\tens\eps\tens 1):
\\
Ts\cc\tens Ts\cc\to Ts\cc\tens Ts\cc.
\end{multline*}
    \else
\[ (\Delta_0\tens 1)(1\tens\eps\tens 1)=1\tens1
=(1\tens\Delta_0)(1\tens\eps\tens 1):
 Ts\cc\tens Ts\cc\to Ts\cc\tens Ts\cc.
\]
    \fi
This computation shows that \(\nu: Ts\cc\tens Ts\cc\to Ts\cc\) is a
double \((1,1)\)\n-coderivation. Its only non-vanishing components are
\(\sS{_{X,Y}}\nu_{1,0}=1:s\cc(X,Y)\to s\cc(X,Y)\) and
\(\sS{_{X,Y}}\nu_{0,1}=1:s\cc(X,Y)\to s\cc(X,Y)\), \(X,Y\in\Ob\cc\).

Since \(\nu B_1\) is a double \((1,1)\)\n-coderivation of degree 1, the
equation \(\nu B_1=0\) is equivalent to its particular case \(\nu
B_1\pr_1=0\), i.e., for each \(n,m\ge0\)
    \ifx\chooseClass1
\begin{multline*}
\sum_{\substack{0\le i\le n,\\ 0\le j\le m}}
(1^{\tens n-i}\tens\nu_{i,j}\tens 1^{\tens m-j})b_{n-i+1+m-j}
\\
\hspace*{3em}
-\sum_{a+k+c=n}(1^{\tens a}\tens b_k\tens 1^{\tens c+m})\nu_{a+1+c,m}
\hfill
\\
\hspace*{3em}
-\sum_{u+t+v=m}(1^{\tens n+u}\tens b_t\tens 1^{\tens v})
\nu_{n,u+1+v}=0:
\hfill
\\
T^ns\cc\tens T^ms\cc\to s\cc.
\end{multline*}
    \else
\begin{multline*}
\sum_{\substack{0\le i\le n,\\ 0\le j\le m}}
(1^{\tens n-i}\tens\nu_{i,j}\tens 1^{\tens m-j})b_{n-i+1+m-j}
-\sum_{a+k+c=n}(1^{\tens a}\tens b_k\tens 1^{\tens c+m})\nu_{a+1+c,m}
\\
-\sum_{u+t+v=m}(1^{\tens n+u}\tens b_t\tens 1^{\tens v})
\nu_{n,u+1+v}=0:
T^ns\cc\tens T^ms\cc\to s\cc.
\end{multline*}
    \fi
It reduces to the identity
    \ifx\chooseClass1
\begin{multline*}
\chi(n>0)b_{n+m}-\chi(m>0)b_{n+m}
\\
-\chi(m=0)b_n+\chi(n=0)b_m=0,
\end{multline*}
    \else
\[ \chi(n>0)b_{n+m}-\chi(m>0)b_{n+m} -\chi(m=0)b_n+\chi(n=0)b_m=0,
\]
    \fi
where \(\chi(P)=1\) if a condition \(P\) holds and \(\chi(P)=0\) if
\(P\) does not hold.
\end{proof}

Let $\cc$ be a strictly unital \ainf-category. The strict unit
$\uni^\cc_0$ is viewed as a morphism of graded quivers
$\uni^\cc_0:\kk\Ob\cc\to s\cc$ of degree $-1$, identity on objects. For
each \(n\ge0\), introduce a morphism of graded quivers
    \ifx\chooseClass1
\begin{multline*}
\xi_n=\bigl[(Ts\cc)^{\tens n+1}
\rto{1\tens\uni^\cc_0\tens1\tdt\uni^\cc_0\tens1}
\\
Ts\cc\tens s\cc\tens Ts\cc\tdt s\cc\tens Ts\cc \rto{\mu^{(2n+1)}}
Ts\cc\bigr],
\end{multline*}
    \else
\[ \xi_n=\bigl[(Ts\cc)^{\tens n+1}
\rto{1\tens\uni^\cc_0\tens1\tdt\uni^\cc_0\tens1}
 Ts\cc\tens s\cc\tens Ts\cc\tdt s\cc\tens Ts\cc \rto{\mu^{(2n+1)}}
Ts\cc\bigr],
\]
    \fi
of degree \(-n\), identity on objects. Here \(\mu^{(2n+1)}\) denotes
composition of \(2n+1\) composable arrows in the graded category
\(Ts\cc\). In particular, \(\xi_0=1:Ts\cc\to Ts\cc\). Denote
 \(\xi=\xi_1=(1\tens\uni^\cc_0\tens1)\mu^{(3)}:
 Ts\cc\tens Ts\cc\to Ts\cc\) for the sake of brevity.

\begin{lemma}\label{lem-xi-double-coder}
The map \(\xi:Ts\cc\tens Ts\cc\to Ts\cc\) is a double
\((1,1)\)\n-coderivation of degree~$-1$ and \(\xi B_1=\nu\).
\end{lemma}

\begin{proof}
The following identity follows directly from the definitions of \(\mu\)
and \(\Delta_0\):
    \ifx\chooseClass1
\begin{multline*}
\mu\Delta_0=(\Delta_0\tens1)(1\tens\mu)+(1\tens\Delta_0)(\mu\tens1)-1:
\\
Ts\cc\tens Ts\cc\to Ts\cc\tens Ts\cc.
\end{multline*}
    \else
\[ \mu\Delta_0=(\Delta_0\tens1)(1\tens\mu)+(1\tens\Delta_0)(\mu\tens1)-1:
Ts\cc\tens Ts\cc\to Ts\cc\tens Ts\cc.
\]
    \fi
It implies
\begin{multline}\label{equ-mu-Delta-relation}
\mu^{(3)}\Delta_0=(\Delta_0\tens1\tens1)(1\tens\mu^{(3)})
+(1\tens1\tens\Delta_0)(\mu^{(3)}\tens1)
\\
    \ifx\chooseClass1
\hskip\multlinegap\hphantom{\mu^{(3)}\Delta_0}
+(1\tens\Delta_0\tens1)(\mu\tens\mu) -(1\tens\mu)-(\mu\tens1):
\hfill
\\
    \else
+(1\tens\Delta_0\tens1)(\mu\tens\mu) -(1\tens\mu)-(\mu\tens1):
    \fi
Ts\cc\tens Ts\cc\tens Ts\cc\to Ts\cc\tens Ts\cc.
\end{multline}
Since
 \(\uni^\cc_0\Delta_0=\uni^\cc_0\tens\eta+\eta\tens\uni^\cc_0:
 \kk\Ob\cc\to Ts\cc\tens Ts\cc\),
it follows that
\begin{multline*}
(1\tens\uni^\cc_0\Delta_0\tens1)(\mu\tens\mu)
-(1\tens(\uni^\cc_0\tens1)\mu) -((1\tens\uni^\cc_0)\mu\tens1)=0:
\\
Ts\cc\tens Ts\cc\to Ts\cc\tens Ts\cc.
\end{multline*}
Equation~\eqref{equ-mu-Delta-relation} yields
\begin{multline*}
(1\tens\uni^\cc_0\tens1)\mu^{(3)}\Delta_0
    \ifx\chooseClass1
\\
    \fi
=(\Delta_0\tens1)(1\tens(1\tens\uni^\cc_0\tens1)\mu^{(3)})
+(1\tens\Delta_0)((1\tens\uni^\cc_0\tens1)\mu^{(3)}\tens1),
\end{multline*}
i.e., \(\xi=(1\tens\uni^\cc_0\tens1)\mu^{(3)}:Ts\cc\tens Ts\cc\to
Ts\cc\) is a double \((1,1)\)\n-coderivation. Its the only
non-vanishing components are \(\sS{_X}\xi_{0,0}=\sS{_X}\uni^\cc_0\in
s\cc(X,X)\), \(X\in\Ob\cc\).

Since both \(\xi B_1\) and \(\nu\) are double \((1,1)\)\n-coderivations
of degree 0, the equation \(\xi B_1=\nu\) is equivalent to its
particular case \(\xi B_1\pr_1=\nu\pr_1\), i.e., for each \(n,m\ge0\)
    \ifx\chooseClass1
\begin{multline*}
\sum_{\substack{0\le p\le n\\ 0\le q\le m}}
(1^{\tens n-p}\tens\xi_{p,q}\tens1^{\tens m-q})b_{n-p+1+m-q}
\\
\hspace*{2em}
+\sum_{a+k+c=n}(1^{\tens a}\tens b_k\tens 1^{\tens c+m})\xi_{a+1+c,m}
\hfill
\\
\hspace*{2em}
+\sum_{u+t+v=m}(1^{\tens n+u}\tens b_t\tens 1^{\tens v})
\xi_{n,u+1+v}=\nu_{n,m}:
\hfill
\\
T^ns\cc\tens T^ms\cc\to s\cc.
\end{multline*}
    \else
\begin{multline*}
\sum_{\substack{0\le p\le n\\ 0\le q\le m}}
(1^{\tens n-p}\tens\xi_{p,q}\tens1^{\tens m-q})b_{n-p+1+m-q}
+\sum_{a+k+c=n}(1^{\tens a}\tens b_k\tens 1^{\tens c+m})\xi_{a+1+c,m}
\\
+\sum_{u+t+v=m}(1^{\tens n+u}\tens b_t\tens 1^{\tens v})
\xi_{n,u+1+v}=\nu_{n,m}:
T^ns\cc\tens T^ms\cc\to s\cc.
\end{multline*}
    \fi
It reduces to the the equation
\[
(1^{\tens n}\tens\uni^\cc_0\tens1^{\tens m})b_{n+1+m} =\nu_{n,m}:
T^ns\cc\tens T^ms\cc\to s\cc,
\]
which holds true, since $\uni^\cc_0$ is a strict unit.
\end{proof}

Note that the maps \(\nu_n\), \(\xi_n\) obey the following relations:
\begin{equation}\label{equ-xi-nu-recursive}
\xi_n=(\xi_{n-1}\tens1)\xi,\qquad
\nu_n =(1^{\tens n}\tens\eps) -(\nu_{n-1}\tens1), \qquad n\ge1.
\end{equation}
In particular, \(\xi_n\eps=0:(Ts\cc)^{\tens n+1}\to\kk\Ob\cc\), for
each \(n\ge1\), as \(\xi\eps=0\) by
equation~\eqref{equ-r-Delta-iterated}.

\begin{lemma}
The following equations hold true:
\begin{alignat}{3}
\xi_n\Delta_0
=\sum_{i=0}^n(1^{\tens i}\tens\Delta_0\tens 1^{\tens n-i})(\xi_i\tens\xi_{n-i}),
&\quad & n\ge0, \label{equ-xin-Delta}
\\
\xi_nb-(-)^n\sum_{i=0}^n(1^{\tens i}\tens b\tens 1^{\tens n-i})\xi_n
=\nu_n\xi_{n-1}, &\quad & n\ge1. \label{equ-xin-b}
\end{alignat}
\end{lemma}

\begin{proof}
Let us prove~\eqref{equ-xin-Delta}. The proof is by induction on $n$.
The case $n=0$ is trivial. Let $n\ge1$. By~\eqref{equ-xi-nu-recursive}
and \lemref{lem-xi-double-coder},
\[
\xi_n\Delta_0=(\xi_{n-1}\tens1)\xi\Delta_0
=(\xi_{n-1}\Delta_0\tens1)(1\tens\xi)+
(\xi_{n-1}\tens\Delta_0)(\xi\tens1).
\]
By induction hypothesis,
\[
\xi_{n-1}\Delta_0=\sum_{i=0}^{n-1}
 (1^{\tens i}\tens\Delta_0\tens1^{\tens n-1-i})(\xi_i\tens\xi_{n-1-i}),
\]
therefore
\begin{multline*}
\xi_n\Delta_0=
\sum_{i=0}^{n-1}(1^{\tens i}\tens\Delta_0\tens 1^{\tens n-i})
(\xi_i\tens\xi_{n-1-i}\tens1)(1\tens\xi)
    \ifx\chooseClass1
\\
\hfill
    \fi
+(1^{\tens n}\tens\Delta_0)((\xi_{n-1}\tens1)\xi\tens1)
\quad
\\
\hskip\multlinegap\hphantom{\xi_n\Delta_0}
=\sum_{i=0}^n(1^{\tens i}\tens\Delta_0\tens1^{\tens n-i})
(\xi_i\tens\xi_{n-i}),\hfill
\end{multline*}
since $(\xi_{n-1-i}\tens1)\xi=\xi_{n-i}$ if $0\le i\le n-1$.

Let us prove \eqref{equ-xin-b}. The proof is by induction on $n$. The
case \(n=1\) follows from \lemref{lem-xi-double-coder}. Let $n\ge2$. By
\eqref{equ-xi-nu-recursive} and \lemref{lem-xi-double-coder},
\begin{multline*}
\xi_n b-(-)^n\sum_{i=0}^n(1^{\tens i}\tens b\tens 1^{\tens n-i})\xi_n
\\
\hskip\multlinegap
=(\xi_{n-1}\tens1)\xi b
-(-)^n\sum_{i=0}^{n-1}((1^{\tens i}\tens b\tens1^{\tens n-1-i})\xi_{n-1}\tens1)\xi
    \ifx\chooseClass1
\hfill
\\
\hfill
    \fi
-(-)^n(1^{\tens n}\tens b)(\xi_{n-1}\tens1)\xi
\quad
\\
\hskip\multlinegap
=-(\xi_{n-1}b\tens1)\xi-(\xi_{n-1}\tens b)\xi+(\xi_{n-1}\tens1)\nu
\hfill
\\
\hfill
+(-)^{n-1}\sum_{i=0}^{n-1}((1^{\tens i}\tens b\tens1^{\tens n-1-i})
\xi_{n-1}\tens1)\xi
+(\xi_{n-1}\tens b)\xi
\quad
\\
\hskip\multlinegap
=(\xi_{n-1}\tens1)\nu
    \ifx\chooseClass1
\hfill
\\
    \fi
-\biggl(\Bigl[\xi_{n-1}b
-(-)^{n-1}\sum_{i=0}^{n-1}(1^{\tens i}\tens b\tens1^{\tens n-1-i})
\xi_{n-1}\Bigr]\tens1\biggr)\xi.
\hfill
\end{multline*}
By induction hypothesis
\[
\xi_{n-1}b
 -(-)^{n-1}\sum_{i=0}^{n-1}(1^{\tens i}\tens b\tens1^{\tens
 n-1-i})\xi_{n-1} =\nu_{n-1}\xi_{n-2},
\]
therefore
\[
\xi_n b-(-)^n\sum_{i=0}^n(1^{\tens i}\tens b\tens 1^{\tens n-i})\xi_n
=(\xi_{n-1}\tens1)\nu-(\nu_{n-1}\xi_{n-2}\tens1)\xi.
\]
Since by \eqref{equ-xi-nu-recursive},
    \ifx\chooseClass1
\begin{multline*}
(\xi_{n-1}\tens1)\nu-(\nu_{n-1}\xi_{n-2}\tens1)\xi
\\
=(\xi_{n-1}\tens\eps)-(\xi_{n-1}\eps\tens1)-(\nu_{n-1}\tens1)\xi_{n-1}
\\
=(1^{\tens n}\tens\eps)\xi_{n-1}-(\nu_{n-1}\tens1)\xi_{n-1}
=\nu_n\xi_{n-1},
\end{multline*}
    \else
\begin{align*}
(\xi_{n-1}\tens1)\nu-(\nu_{n-1}\xi_{n-2}\tens1)\xi
&=(\xi_{n-1}\tens\eps)-(\xi_{n-1}\eps\tens1)-(\nu_{n-1}\tens1)\xi_{n-1}
\\
&=(1^{\tens n}\tens\eps)\xi_{n-1}-(\nu_{n-1}\tens1)\xi_{n-1}
=\nu_n\xi_{n-1},
\end{align*}
    \fi
equation~\eqref{equ-xin-b} is proven.
\end{proof}

\section{An augmented differential graded cocategory}
\label{sec-proof-unital-implies-unital-structure}

Let now $\cc=\ca^\su$, where \(\ca\) is an \ainf-category. There is an
isomorphism of graded $\kk$-quivers, identity on objects:
\[
\zeta:\bigoplus_{n\ge0}(Ts\ca)^{\tens n+1}[n]\to Ts\ca^\su.
\]
The morphism $\zeta$ is the sum of morphisms
    \ifx\chooseClass1
\begin{multline}\label{equ-def-zeta}
\zeta_n=\bigl[(Ts\ca)^{\tens n+1}[n] \rto{s^{-n}} (Ts\ca)^{\tens n+1}
\\
\xhookrightarrow{e^{\tens n+1}} (Ts\ca^\su)^{\tens n+1} \rto{\xi_n}
Ts\ca^\su\bigr],
\end{multline}
    \else
\begin{equation}\label{equ-def-zeta}
\zeta_n=\bigl[(Ts\ca)^{\tens n+1}[n] \rto{s^{-n}} (Ts\ca)^{\tens n+1}
\xhookrightarrow{e^{\tens n+1}} (Ts\ca^\su)^{\tens n+1} \rto{\xi_n}
Ts\ca^\su\bigr],
\end{equation}
    \fi
where \(e:\ca\hookrightarrow\ca^\su\) is the natural embedding. The
graded quiver
\[
\ce\defeq\bigoplus_{n\ge0}(Ts\ca)^{\tens n+1}[n]
\]
admits a unique structure of an augmented differential graded
cocategory such that $\zeta$ becomes an isomorphism of augmented
differential graded cocategories. The comultiplication
$\wt\Delta:\ce\to\ce\tens\ce$ is found from the equation
    \ifx\chooseClass1
\begin{multline*}
\bigl[\ce \rto{\zeta} Ts\ca^\su \rto{\Delta_0}
Ts\ca^\su\tens Ts\ca^\su\bigr]\\
=\bigl[\ce \rto{\wt\Delta} \ce\tens\ce
\rto{\zeta\tens\zeta} Ts\ca^\su\tens Ts\ca^\su\bigr].
\end{multline*}
    \else
\begin{equation*}
\bigl[\ce \rto{\zeta} Ts\ca^\su \rto{\Delta_0}
Ts\ca^\su\tens Ts\ca^\su\bigr]
=\bigl[\ce \rto{\wt\Delta} \ce\tens\ce
\rto{\zeta\tens\zeta} Ts\ca^\su\tens Ts\ca^\su\bigr].
\end{equation*}
    \fi
Restricting the left hand side of the equation to the summand
$(Ts\ca)^{\tens n+1}[n]$ of \(\ce\), we obtain
    \ifx\chooseClass1
\begin{multline*}
\zeta_n\Delta_0=s^{-n}e^{\tens n+1}\xi_n\Delta_0
\\
\hskip\multlinegap\hphantom{\zeta_n\Delta_0}
=s^{-n}\sum_{i=0}^n(e^{\tens i}\tens e\Delta_0\tens e^{\tens n-i})
(\xi_i\tens\xi_{n-i}):
\hfill
\\
(Ts\ca)^{\tens n+1}[n]\to Ts\ca^\su\tens Ts\ca^\su,
\end{multline*}
    \else
\begin{align*}
\zeta_n\Delta_0 &=s^{-n}e^{\tens n+1}\xi_n\Delta_0
\\
&=s^{-n}\sum_{i=0}^n(e^{\tens i}\tens e\Delta_0\tens e^{\tens n-i})
(\xi_i\tens\xi_{n-i}):
(Ts\ca)^{\tens n+1}[n]\to Ts\ca^\su\tens Ts\ca^\su,
\end{align*}
    \fi
by equation~\eqref{equ-xin-Delta}. Since $e$ is a morphism of augmented
graded cocategories, it follows that
\begin{multline*}
\zeta_n\Delta_0
=s^{-n}\sum_{i=0}^n(1^{\tens i}\tens\Delta_0\tens1^{\tens n-i})
(e^{\tens i+1}\xi_i\tens e^{\tens n-i+1}\xi_{n-i})
\\
\hskip\multlinegap\hphantom{\zeta_n\Delta_0}
=s^{-n}\sum_{i=0}^n(1^{\tens i}\tens\Delta_0\tens1^{\tens n-i})
(s^i\tens s^{n-i})(\zeta_i\tens\zeta_{n-i}):
\hfill
\\
(Ts\ca)^{\tens n+1}[n]\to Ts\ca^\su\tens Ts\ca^\su.
\end{multline*}
This implies the following formula for $\wt\Delta$:
\begin{multline}\label{equ-tilde-Delta-explicit}
\wt\Delta|_{(Ts\ca)^{\tens n+1}[n]}
=s^{-n}\sum_{i=0}^n(1^{\tens i}\tens\Delta_0\tens1^{\tens n-i})(s^i\tens s^{n-i}):
\\
(Ts\ca)^{\tens n+1}[n]\to\bigoplus_{i=0}^n(Ts\ca)^{\tens i+1}[i]
\bigotimes(Ts\ca)^{\tens n-i+1}[n-i].
\end{multline}
The counit of $\ce$ is
 $\wt\eps=[\ce \rto{\pr_0} Ts\ca \rto\eps \kk\Ob\ca=\kk\Ob\ce]$. The
augmentation of \(\ce\) is
\(\wt\eta=[\kk\Ob\ce=\kk\Ob\ca\rto{\eta}Ts\ca\rto{\inj_0}\ce]\). The
differential \(\wt b:\ce\to\ce\) is found from the following equation:
\[
\bigl[\ce \rto{\zeta} Ts\ca^\su \rto{b} Ts\ca^\su\bigr]
=\bigl[\ce\rto{\wt b} \ce \rto\zeta Ts\ca^\su\bigr].
\]
Let \(\wt b_{n,m}:(Ts\ca)^{\tens n+1}[n]\to(Ts\ca)^{\tens m+1}[m]\),
\(n,m\ge0\), denote the matrix coefficients of \(\wt b\). Restricting
the left hand side of the above equation to the summand
\((Ts\ca)^{\tens n+1}[n]\) of \(\ce\), we obtain
    \ifx\chooseClass1
\begin{multline*}
\zeta_n b=s^{-n}e^{\tens n+1}\xi_n b
\\
\hskip\multlinegap\hphantom{\zeta_n b}
=s^{-n}e^{\tens
n+1}\nu_n\xi_{n-1} +(-)^ns^{-n}\sum_{i=0}^n(e^{\tens i}\tens eb\tens
e^{\tens n-i})\xi_n:
\hfill
\\
(Ts\ca)^{\tens n+1}[n]\to Ts\ca^\su,
\end{multline*}
    \else
\begin{align*}
\zeta_n b &=s^{-n}e^{\tens n+1}\xi_n b
\\
&=s^{-n}e^{\tens
n+1}\nu_n\xi_{n-1} +(-)^ns^{-n}\sum_{i=0}^n(e^{\tens i}\tens eb\tens
e^{\tens n-i})\xi_n:
(Ts\ca)^{\tens n+1}[n]\to Ts\ca^\su,
\end{align*}
    \fi
by equation~\eqref{equ-xin-b}. Since \(e\) preserves the counit, it
follows that
\[
e^{\tens n+1}\nu_n=\nu_n e^{\tens n}:(Ts\ca)^{\tens
n+1}\to(Ts\ca^\su)^{\tens n}.
\]
Furthermore, \(e\) commutes with the differential \(b\), therefore
    \ifx\chooseClass1
\begin{multline*}
\zeta_n b=s^{-n}\nu_n s^{n-1}(s^{-(n-1)}e^{\tens n}\xi_{n-1})
\\
\hfill
+(-)^ns^{-n}\sum_{i=0}^n(1^{\tens i}\tens b\tens 1^{\tens n-i})
s^n(s^{-n}e^{\tens n+1}\xi_n)
\quad
\\
\hskip\multlinegap\hphantom{\zeta_n b}
=s^{-n}\nu_n s^{n-1}\zeta_{n-1} +(-)^ns^{-n}\sum_{i=0}^n(1^{\tens i}
\tens b\tens 1^{\tens n-i})s^n\zeta_n:
\hfill
\\
(Ts\ca)^{\tens n+1}[n]\to Ts\ca^\su.
\end{multline*}
    \else
\begin{align*}
\zeta_n b &=s^{-n}\nu_n s^{n-1}(s^{-(n-1)}e^{\tens n}\xi_{n-1})
+(-)^ns^{-n}\sum_{i=0}^n(1^{\tens i}\tens b\tens 1^{\tens n-i})
s^n(s^{-n}e^{\tens n+1}\xi_n)
\\
&=s^{-n}\nu_n s^{n-1}\zeta_{n-1} +(-)^ns^{-n}\sum_{i=0}^n(1^{\tens i}
\tens b\tens 1^{\tens n-i})s^n\zeta_n:
(Ts\ca)^{\tens n+1}[n]\to Ts\ca^\su.
\end{align*}
    \fi
We conclude that
    \ifx\chooseClass1
\begin{multline}
\wt b_{n,n}=(-)^ns^{-n}\sum_{i=0}^n(1^{\tens i}\tens b\tens1^{\tens n-i})s^n:
\\
(Ts\ca)^{\tens n+1}[n]\to(Ts\ca)^{\tens n+1}[n],
\label{equ-tilde-b-explicit-n-n}
\end{multline}
    \else
\begin{equation}
\wt b_{n,n}=(-)^ns^{-n}\sum_{i=0}^n(1^{\tens i}\tens b\tens1^{\tens n-i})s^n:
(Ts\ca)^{\tens n+1}[n]\to(Ts\ca)^{\tens n+1}[n],
\label{equ-tilde-b-explicit-n-n}
\end{equation}
    \fi
for \(n\ge 0\), and
\begin{equation}
\wt b_{n,n-1}=s^{-n}\nu_ns^{n-1}:
(Ts\ca)^{\tens n+1}[n]\to(Ts\ca)^{\tens n}[n-1],
\label{equ-tilde-b-explicit-n-n-1}
\end{equation}
for \(n\ge1\), are the only non-vanishing matrix coefficients of \(\wt
b\).

Let \(g:\ce\to Ts\cb\) be a morphism of augmented differential graded
cocategories, and let \(g_n:(Ts\ca)^{\tens n+1}[n]\to Ts\cb\) be its
components. By formula~\eqref{equ-tilde-Delta-explicit}, the equation
\(g\Delta_0=\wt\Delta(g\tens g)\) is equivalent to the system of
equations
\begin{multline*}
g_n\Delta_0=s^{-n}\sum_{i=0}^n(1^{\tens i}\tens\Delta_0\tens 1^{\tens
n-i}) (s^ig_i\tens s^{n-i}g_{n-i}):
\\
(Ts\ca)^{\tens n+1}[n]\to Ts\cb\tens Ts\cb,\quad n\ge 0.
\end{multline*}
The equation \(g\eps=\wt\eps(\kk\Ob g)\) is equivalent to the equations
\(g_0\eps=\eps(\kk\Ob g_0)\), \(g_n\eps=0\), \(n\ge1\). The equation
\(\wt\eta g=(\kk\Ob g)\eta\) is equivalent to the equation \(\eta
g_0=(\kk\Ob g_0)\eta\). By formulas~\eqref{equ-tilde-b-explicit-n-n}
and \eqref{equ-tilde-b-explicit-n-n-1}, the equation \(gb=\wt b g\) is
equivalent to \(g_0b=bg_0:Ts\ca\to Ts\cb\) and
\begin{multline*}
g_nb=(-)^ns^{-n}\sum_{i=0}^n(1^{\tens i}\tens b\tens 1^{\tens n-i})
s^ng_n+s^{-n}\nu_ns^{n-1}g_{n-1}:
\\
(Ts\ca)^{\tens n+1}[n]\to Ts\cb, \quad n\ge1.
\end{multline*}
Introduce \(\kk\)\n-linear maps \(\phi_n=s^ng_n:(Ts\ca)^{\tens
n+1}(X,Y)\to Ts\cb(Xg,Yg)\) of degree \(-n\), \(X,Y\in\Ob\ca\),
\(n\ge0\). The above equations take the following form:
    \ifx\chooseClass1
\begin{multline}
\phi_n\Delta_0=\sum_{i=0}^n(1^{\tens i}\tens\Delta_0\tens 1^{\tens n-i})
(\phi_i\tens\phi_{n-i}):
\\
(Ts\ca)^{\tens n+1}\to Ts\cb\tens Ts\cb,
\label{equ-compat-with-Delta}
\end{multline}
    \else
\begin{equation}
\phi_n\Delta_0=\sum_{i=0}^n(1^{\tens i}\tens\Delta_0\tens 1^{\tens n-i})
(\phi_i\tens\phi_{n-i}):
(Ts\ca)^{\tens n+1}\to Ts\cb\tens Ts\cb,
\label{equ-compat-with-Delta}
\end{equation}
    \fi
for \(n\ge 1\);
    \ifx\chooseClass1
\begin{multline}
\phi_nb=(-)^n\sum_{i=0}^n(1^{\tens i}\tens b\tens 1^{\tens n-i})
\phi_n+\nu_n\phi_{n-1}:
\\
(Ts\ca)^{\tens n+1}\to Ts\cb,\label{equ-compat-with-b}
\end{multline}
    \else
\begin{equation}
\phi_nb=(-)^n\sum_{i=0}^n(1^{\tens i}\tens b\tens 1^{\tens n-i})
\phi_n+\nu_n\phi_{n-1}:
(Ts\ca)^{\tens n+1}\to Ts\cb,\label{equ-compat-with-b}
\end{equation}
    \fi
for \(n\ge1\);
\begin{gather}
\phi_0\Delta_0=\Delta_0(\phi_0\tens\phi_0),\quad
\phi_0\eps=\eps,\quad \phi_0b=b\phi_0,
\\
\phi_n\eps=0,\quad n\ge 1. \label{equ-compat-with-epsilon}
\end{gather}
Summing up, we conclude that morphisms of augmented differential graded
cocategories \(\ce\to Ts\cb\) are in bijection with collections
consisting of a morphism of augmented differential graded cocategories
\(\phi_0:Ts\ca\to Ts\cb\) and of \(\kk\)\n-linear maps
\(\phi_n:(Ts\ca)^{\tens n+1}(X,Y)\to Ts\cb(X\phi_0,Y\phi_0)\) of degree
\(-n\), \(X,Y\in\Ob\ca\), \(n\ge1\), such that
equations~\eqref{equ-compat-with-Delta}, \eqref{equ-compat-with-b}, and
\eqref{equ-compat-with-epsilon} hold true.

In particular, \ainf-functors \(f:\ca^\su\to\cb\), which are augmented
differential graded cocategory morphisms \(Ts\ca^\su\to Ts\cb\), are in
bijection with morphisms
 \(g=\zeta f:\ce\to Ts\cb\) of augmented differential graded
cocategories. With the above notation, we may say that to give an
\ainf-functor \(f:\ca^\su\to\cb\) is the same as to give an
\ainf-functor \(\phi_0:\ca\to\cb\) and a system of \(\kk\)\n-linear
maps \(\phi_n:(Ts\ca)^{\tens n+1}(X,Y)\to Ts\cb(X\phi_0,Y\phi_0)\) of
degree \(-n\), \(X,Y\in\Ob\ca\), \(n\ge1\), such that
equations~\eqref{equ-compat-with-Delta}, \eqref{equ-compat-with-b} and
\eqref{equ-compat-with-epsilon} hold true.

\begin{proposition}\label{prop-existence-unital-structure}
The following conditions are equivalent.
\begin{itemize}
\item[(a)] There exists an \ainf-functor \(U:\ca^\su\to\ca\) such that
\[
\bigl[\ca \xhookrightarrow{e} \ca^\su \rto{U} \ca\bigr]=\id_\ca.
\]
\item[(b)] There exists a double \((1,1)\)\n-coderivation
\(\phi:Ts\ca\tens Ts\ca\to Ts\ca\) of degree \(-1\) such that
\(\phi B_1=\nu\).
\end{itemize}
\end{proposition}

\begin{proof}
(a)\(\Rightarrow\)(b) Let \(U:\ca^\su\to\ca\) be an \ainf-functor such
that \(eU=\id_\ca\), in particular
\(\Ob U=\id:\Ob\ca^\su=\Ob\ca\to\Ob\ca\). It gives rise to the family
of \(\kk\)\n-linear maps
\(\phi_n=s^n\zeta_n U:(Ts\ca)^{\tens n+1}(X,Y)\to Ts\cb(X,Y)\) of
degree \(-n\), \(X,Y\in\Ob\ca\), \(n\ge0\), that satisfy
equations~\eqref{equ-compat-with-Delta}, \eqref{equ-compat-with-b} and
\eqref{equ-compat-with-epsilon}. In particular, \(\phi_0=eU=\id_\ca\).
Equations~\eqref{equ-compat-with-Delta} and \eqref{equ-compat-with-b}
for \(n=1\) read as follows:
    \ifx\chooseClass1
\begin{multline*}
\phi_1\Delta_0=(\Delta_0\tens1)(\phi_0\tens\phi_1)
+(1\tens\Delta_0)(\phi_1\tens\phi_0)
\\
\hfill
=(\Delta_0\tens1)(1\tens\phi_1)+(1\tens\Delta_0)(\phi_1\tens1),
\quad
\\
\hskip\multlinegap
\phi_1b=(1\tens
b+b\tens1)\phi_1+\nu_1\phi_0 =(1\tens b+b\tens1)\phi_1+\nu.
\hfill
\end{multline*}
    \else
\begin{align*}
\phi_1\Delta_0 &=(\Delta_0\tens1)(\phi_0\tens\phi_1)
+(1\tens\Delta_0)(\phi_1\tens\phi_0)
=(\Delta_0\tens1)(1\tens\phi_1)+(1\tens\Delta_0)(\phi_1\tens1),
\\
\phi_1b &=(1\tens b+b\tens1)\phi_1+\nu_1\phi_0 =(1\tens
b+b\tens1)\phi_1+\nu.
\end{align*}
    \fi
In other words, \(\phi_1\) is a double \((1,1)\)\n-coderivation of
degree \(-1\) and \(\phi_1B_1=\nu\).

(b)\(\Rightarrow\)(a) Let \(\phi:Ts\ca\tens Ts\ca\to Ts\ca\) be a
double \((1,1)\)-coderivation of degree \(-1\) such that \(\phi
B_1=\nu\). Define \(\kk\)\n-linear maps
\[
\phi_n:(Ts\ca)^{\tens n+1}(X,Y)\to Ts\ca(X,Y), \quad X,Y\in\Ob\ca,
\]
of degree \(-n\), \(n\ge0\), recursively via \(\phi_0=\id_\ca\) and
\(\phi_n=(\phi_{n-1}\tens1)\phi\), \(n\ge1\).  Let us show that \(\phi_n\) satisfy
equations \eqref{equ-compat-with-Delta}, \eqref{equ-compat-with-b} and
\eqref{equ-compat-with-epsilon}.
Equation~\eqref{equ-compat-with-epsilon} is obvious:
\(\phi_n\eps=(\phi_{n-1}\tens1)\phi\eps=0\) as \(\phi\eps=0\) by
\eqref{equ-r-Delta-iterated}. Let us prove
equation~\eqref{equ-compat-with-Delta} by induction. It holds for
\(n=1\) by assumption, since \(\phi_1=\phi\) is a double
\((1,1)\)\n-coderivation. Let \(n\ge2\). We have:
\begin{multline*}
\phi_n\Delta_0=(\phi_{n-1}\tens1)\phi_1\Delta_0
\\
\hskip\multlinegap\hphantom{\phi_n\Delta_0}
=(\phi_{n-1}\tens1)((\Delta_0\tens1)(1\tens\phi_1)+(1\tens\Delta_0)(\phi_1\tens1))
\hfill
\\
\hskip\multlinegap\hphantom{\phi_n\Delta_0}
=(\phi_{n-1}\Delta_0\tens1)(1\tens\phi_1)
    \ifx\chooseClass1
\hfill
\\
+(1^{\tens n}\tens\Delta_0)((\phi_{n-1}\tens1)\phi_1\tens1).
    \else
+(1^{\tens n}\tens\Delta_0)((\phi_{n-1}\tens1)\phi_1\tens1).
\hfill
    \fi
\end{multline*}
By induction hypothesis,
\[
\phi_{n-1}\Delta_0=\sum_{i=0}^{n-1}
 (1^{\tens i}\tens\Delta_0\tens1^{\tens n-1-i})
 (\phi_i\tens\phi_{n-1-i}),
\]
so that
\begin{multline*}
\phi_n\Delta_0=\sum_{i=0}^{n-1}
 (1^{\tens i}\tens\Delta_0\tens1^{\tens n-i})
 (\phi_i\tens\phi_{n-1-i}\tens1)(1\tens\phi_1)
\\
\hfill
+(1^{\tens n}\tens\Delta_0)((\phi_{n-1}\tens1)\phi_1\tens1)
\quad
\\
\hskip\multlinegap\hphantom{\phi_n\Delta_0}
=\sum_{i=0}^n(1^{\tens i}\tens\Delta_0\tens1^{\tens n-i}) (\phi_i\tens\phi_{n-i}),
\hfill
\end{multline*}
since $(\phi_{n-1-i}\tens1)\phi_1=\phi_{n-i}$, $0\le i\le n-1$.

Let us prove equation~\eqref{equ-compat-with-b} by induction. For
\(n=1\) it is equivalent to the equation \(\phi B_1=\nu\), which holds
by assumption. Let \(n\ge2\). We have:
\begin{multline*}
\phi_n b-(-)^n\sum_{i=0}^n(1^{\tens i}\tens b\tens 1^{\tens n-i})\phi_n
\\
\hskip\multlinegap
=(\phi_{n-1}\tens1)\phi b
-(-)^n\sum_{i=0}^{n-1}((1^{\tens i}\tens b\tens1^{\tens n-1-i})\phi_{n-1}\tens1)\phi
\hfill
\\
\hfill
-(-)^n(1^{\tens n}\tens b)(\phi_{n-1}\tens1)\phi
\quad
\\
\hskip\multlinegap
=-(\phi_{n-1}b\tens1)\phi-(\phi_{n-1}\tens b)\phi+(\phi_{n-1}\tens1)\nu
\hfill
\\
\hfill
+(-)^{n-1}\sum_{i=0}^{n-1}
 ((1^{\tens i}\tens b\tens1^{\tens n-1-i})\phi_{n-1}\tens1)\phi
+(\phi_{n-1}\tens b)\phi
\quad
\\
\hskip\multlinegap
=(\phi_{n-1}\tens1)\nu
    \ifx\chooseClass1
\hfill
\\
\hskip\multlinegap
    \fi
-\biggl(\Bigl[\phi_{n-1}b
-(-)^{n-1}\sum_{i=0}^{n-1}
(1^{\tens i}\tens b\tens1^{\tens n-1-i})\phi_{n-1}\Bigr]\tens1\biggr)\phi.
\hfill
\end{multline*}
By induction hypothesis,
\[
\phi_{n-1}b-(-)^{n-1}\sum_{i=0}^{n-1}(1^{\tens i}\tens
 b\tens1^{\tens n-1-i})\phi_{n-1}=\nu_{n-1}\phi_{n-2},
\]
therefore
    \ifx\chooseClass1
\begin{multline*}
\phi_n b-(-)^n\sum_{i=0}^n(1^{\tens i}\tens b\tens 1^{\tens n-i})\phi_n
\\
=(\phi_{n-1}\tens1)\nu-(\nu_{n-1}\phi_{n-2}\tens1)\phi.
\end{multline*}
    \else
\[ \phi_n b-(-)^n\sum_{i=0}^n(1^{\tens i}\tens b\tens 1^{\tens n-i})\phi_n
=(\phi_{n-1}\tens1)\nu-(\nu_{n-1}\phi_{n-2}\tens1)\phi.
\]
\fi
Since by \eqref{equ-xi-nu-recursive}
    \ifx\chooseClass1
\begin{multline*}
(\phi_{n-1}\tens1)\nu-(\nu_{n-1}\phi_{n-2}\tens1)\phi
\\
=(\phi_{n-1}\tens\eps)-(\phi_{n-1}\eps\tens1)-(\nu_{n-1}\tens1)\phi_{n-1}
\\
=(1^{\tens n}\tens\eps)\phi_{n-1} -(\nu_{n-1}\tens1)\phi_{n-1}
=\nu_n\phi_{n-1},
\end{multline*}
    \else
\begin{align*}
(\phi_{n-1}\tens1)\nu-(\nu_{n-1}\phi_{n-2}\tens1)\phi
&=(\phi_{n-1}\tens\eps)-(\phi_{n-1}\eps\tens1)-(\nu_{n-1}\tens1)\phi_{n-1}
\\
&=(1^{\tens n}\tens\eps)\phi_{n-1} -(\nu_{n-1}\tens1)\phi_{n-1}
=\nu_n\phi_{n-1},
\end{align*}
    \fi
and equation~\eqref{equ-compat-with-b} is proven.

The system of maps \(\phi_n\), \(n\ge0\), corresponds to an
\ainf-functor \(U:\ca^\su\to\ca\) such that \(\phi_n=s^n\zeta_n U\),
\(n\ge0\). In particular, \(eU=\phi_0=\id_\ca\).
\end{proof}

\begin{proposition}\label{prop-existence-double-coder}
Let \(\ca\) be a unital \ainf-category. There exists a double
\((1,1)\)\n-coderivation \(h:Ts\ca\tens Ts\ca\to Ts\ca\) of degree
\(-1\) such that \(hB_1=\nu\).
\end{proposition}

\begin{proof}
Let $\ca$ be a unital \ainf-category. By
\cite[Corollary~A.12]{math.CT/0306018}, there exist a differential
graded category $\cd$ and an \ainf-equivalence $f:\ca\to\cd$. The
functor $f$ is unital by \cite[Corollary~8.9]{Lyu-AinfCat}. This means
that, for every object $X$ of $\ca$, there exists a $\kk$-linear map
$\sS{_X}v_0:\kk\to(s\cd)^{-2}(Xf,Xf)$ such that
$\sS{_X}\uni^{\ca}_0f_1=\sS{_{Xf}}\uni^\cd_0+\sS{_X}v_0b_1$. Here
$\sS{_{Xf}}\uni^\cd_0$ denotes the strict unit of the differential
graded category $\cd$.

By \lemref{lem-xi-double-coder},
\(\xi=(1\tens\uni^\cd_0\tens1)\mu^{(3)}:Ts\cd\tens Ts\cd\to Ts\cd\) is
a \((1,1)\)\n-coderivation of degree \(-1\). Let \(\iota\) denote the
double \((f,f)\)\n-coderivation \((f\tens f)\xi\) of degree \(-1\). By
\lemref{lem-xi-double-coder},
\[
\iota B_1=(f\tens f)(\xi B_1)=(f\tens f)\nu=\nu f.
\]
By \lemref{lem-nu-double-coder}, the equation \(\nu B_1=0\) holds true.
We conclude that the double coderivations \(\nu\in
\McD(\ca,\ca)(\id_\ca,\id_\ca)^0\) and
\(\iota\in\McD(\ca,\cd)(f,f)^{-1}\) satisfy the following equations:
\begin{align}
\nu B_1&=0,\label{equ-iB}\\
\iota B_1-\nu f&=0.\label{equ-jB-if}
\end{align}
We are going to prove that there exist double coderivations
$h\in\McD(\ca,\ca)(\id_\ca,\id_\ca)^{-1}$ and
$k\in\McD(\ca,\cd)(f,f)^{-2}$ such that the following equations hold
true:
\begin{align*}
hB_1&=\nu,\\
hf&=\iota+kB_1.
\end{align*}
Let us put $\sS{_X}h_{0,0}=\sS{_X}\uni^\ca_0$,
$\sS{_X}k_{0,0}=\sS{_X}v_0$, and construct the other components of $h$
and $k$ by induction. Given an integer $t\ge0$, assume that we have
already found components $h_{p,q}$, $k_{p,q}$ of the sought $h$, $k$,
for all pairs $(p,q)$ with $p+q<t$, such that the equations
    \ifx\chooseClass1
\begin{multline}
(hB_1-\nu)_{p,q}=0:
\\
s\ca(X_0,X_1)\tdt s\ca(X_{p+q-1},X_{p+q})\to
s\ca(X_0,X_{p+q}),\label{eq-h}
\end{multline}
\begin{multline}
(kB_1+\iota-hf)_{p,q}=0:
\\
s\ca(X_0,X_1)\tdt s\ca(X_{p+q-1},X_{p+q})\to
s\cd(X_0f,X_{p+q}f)\label{eq-k}
\end{multline}
    \else
\begin{gather}
(hB_1-\nu)_{p,q}=0:
s\ca(X_0,X_1)\tdt s\ca(X_{p+q-1},X_{p+q})\to
s\ca(X_0,X_{p+q}),\label{eq-h}
\\
(kB_1+\iota-hf)_{p,q}=0:
s\ca(X_0,X_1)\tdt s\ca(X_{p+q-1},X_{p+q})\to
s\cd(X_0f,X_{p+q}f)\label{eq-k}
\end{gather}
    \fi
are satisfied for all pairs $(p,q)$ with $p+q<t$. Introduce double
coderivations $\wt h\in \McD(\ca,\ca)(\id_\ca,\id_\ca)$ and $\wt
k\in\McD(\ca,\cd)(f,f)$ of degree $-1$ resp. $-2$ by their components:
$\wt h_{p,q}=h_{p,q}$, $\wt k_{p,q}=k_{p,q}$ for $p+q<t$, all the other
components vanish. Define a double $(1,1)$-coderivation $\lambda=\wt h
B_1-\nu$ of degree $0$ and a double $(f,f)$-coderivation $\kappa=\wt k
B_1+\iota-\wt h f$ of degree $-1$. Then $\lambda_{p,q}=0$,
$\kappa_{p,q}=0$ for all $p+q<t$. Let non-negative integers $n$, $m$
satisfy \(n+m=t\). The identity $\lambda B_1=0$ implies that
\[
\lambda_{n,m}b_1-\sum_{l=1}^{n+m}
(1^{\tens l-1}\tens b_1\tens 1^{\tens n+m-l})\lambda_{n,m}=0.
\]
The $(n,m)$-component of the identity $\kappa B_1+\lambda f=0$ gives
\[ \kappa_{n,m}b_1 +\sum_{l=1}^{n+m}
(1^{\tens l-1}\tens b_1\tens 1^{\tens n+m-l})\kappa_{n,m}
+\lambda_{n,m}f_1=0.
\]
The chain map $f_1:\ca(X_0,X_{n+m})\to s\cd(X_0f,X_{n+m}f)$ is homotopy
invertible as $f$ is an \ainf-equivalence. Hence, the chain map
\(\Phi\) given by
\begin{align*}
\ucom^\bull(N,s\ca(X_0,X_{n+m}))&\to\ucom^\bull(N,s\cd(X_0f,X_{n+m}f)),
\\
\lambda&\mapsto\lambda f_1,
\end{align*}
is homotopy invertible for each complex of $\kk$-modules $N$, in
particular, for $N=s\ca(X_0,X_1)\tdt s\ca(X_{n+m-1},X_{n+m})$.
Therefore, the complex $\Cone(\Phi)$ is contractible, e.g. by
\cite[Lemma~B.1]{Lyu-AinfCat}. Consider the element
\((\lambda_{n,m},\kappa_{n,m})\) of
\[
\ucom^0(N,s\ca(X_0,X_{n+m}))\oplus\ucom^{-1}(N,\cd(X_0f,X_{n+m}f)).
\]
The above direct sum coincides with \(\Cone^{-1}(\Phi)\). The equations
$-\lambda_{n,m} d=0$, $\kappa_{n,m} d+\lambda_{n,m}\Phi=0$ imply that
\((\lambda_{n,m},\kappa_{n,m})\) is a cycle in the complex
\(\Cone(\Phi)\). Due to acyclicity of $\Cone(\Phi)$,
\((\lambda_{n,m},\kappa_{n,m})\) is a boundary of some element
\((h_{n,m},-k_{n,m})\) of \(\Cone^{-2}(\Phi)\), i.e., of
\[
\ucom^{-1}(N,s\ca(X_0,X_{n+m}))\oplus\ucom^{-2}(N,\cd(X_0f,X_{n+m}f)).
\]
Thus, $-k_{n,m}d+h_{n,m}f_1=\kappa_{n,m}$, $-h_{n,m}d=\lambda_{n,m}$.
These equations can be written as follows:
    \ifx\chooseClass1
\begin{multline*}
-h_{n,m}b_1-\sum_{u+1+v=n+m}(1^{\tens u}\tens b_1\tens 1^{\tens v})h_{n,m}
\\
\hfill
=(\wt h B_1-\nu)_{n,m},
\quad
\\
\hskip\multlinegap
-k_{n,m}b_1+\sum_{u+1+v=n+m}(1^{\tens u}\tens b_1\tens 1^{\tens v})k_{n,m}
+h_{n,m}f_1
\hfill
\\
=(\wt kB_1+\iota-\wt hf)_{n,m}.
\end{multline*}
    \else
\begin{align*}
-h_{n,m}b_1-\sum_{u+1+v=n+m}(1^{\tens u}\tens b_1\tens 1^{\tens v})h_{n,m}
&=(\wt h B_1-\nu)_{n,m},
\\
-k_{n,m}b_1+\sum_{u+1+v=n+m}(1^{\tens u}\tens b_1\tens 1^{\tens v})k_{n,m}
+h_{n,m}f_1
&=(\wt kB_1+\iota-\wt hf)_{n,m}.
\end{align*}
    \fi
Thus, if we introduce double coderivations $\overline h$ and $\overline
k$ by their components: $\overline{h}_{p,q}=h_{p,q}$,
$\overline{k}_{p,q}=k_{p,q}$ for $p+q\le t$ (using just found maps if
$p+q=t$) and $0$ otherwise, then these coderivations satisfy
equations~\eqref{eq-h} and \eqref{eq-k} for each \(p,q\) such that
$p+q\le t$. Induction on $t$ proves the proposition.
\end{proof}

\begin{theorem}\label{thm-unital-structure}
Every unital \ainf-category admits a weak unit.
\end{theorem}

\begin{proof}
The proof follows from Propositions \ref{prop-existence-unital-structure} and
\ref{prop-existence-double-coder}.
\end{proof}

    \ifx\chooseClass1
\section{Summary}
We have proved that the definitions of unital \ainf-category given by
the first author, by Kontsevich and Soibelman, and by Fukaya are
equivalent.
    \fi

\bibliographystyle{amsplain}
%\bibliography{yuri}

\end{document}